\documentclass[12pt]{article}
\usepackage{}
\usepackage[numbers,sort&compress]{natbib}

\usepackage{color}
\usepackage{threeparttable}

\usepackage[english]{babel}
\usepackage{t1enc}
\usepackage{amsthm,amsmath,amssymb}
\usepackage{mathrsfs}
\usepackage[mathscr]{eucal}
\usepackage[latin1]{inputenc}
\usepackage[english]{babel}
\usepackage{latexsym}
\usepackage{graphicx}
\usepackage[noend]{algpseudocode}
\usepackage{algorithmicx,algorithm}
\usepackage{lineno}
\usepackage{booktabs}
\usepackage{multirow}
\newtheorem{theorem}{Theorem}
\newtheorem{corollary}{Corollary}
\newtheorem{proposition}{Proposition}
\newtheorem{lemma}{Lemma}
\newtheorem{remark}{Remark}
\newtheorem{observation}{Observation}
\newtheorem{conjecture}{Conjecture}

\newtheorem{definition}{Definition}

\def\sqrt{\textup{sqrt}}
\def \T{\textup{T}}

\def \diag{\textup{diag}}

\def \Res{\textup{Res}}
\makeatletter
\newcommand{\rmnum}[1]{\romannumeral #1}
\newcommand\restr[2]{{
		\left.\kern-\nulldelimiterspace 
		#1 
		\right|_{#2} 
}}
\newcommand{\Rmnum}[1]{\expandafter\@slowromancap\romannumeral #1@}
 
\makeatother
\title{On the  determinant of the $Q$-walk matrix of rooted product with a  path}
\author{\small Zhidan Yan$^{{\rm a}}$\quad\quad Lihuan Mao$^{\rm b}$\quad\quad Wei Wang$^{\rm a}\thanks{Corresponding author: wangwei.math@gmail.com}$
\\
{\footnotesize$^{\rm a}$School of Mathematics, Physics and Finance, Anhui Polytechnic University, Wuhu 241000, China}\\
{\footnotesize$^{\rm b}$School of Mathematics and Data Science, Shaanxi University of Science and Technology, Xi'an 710021, China}
}
\date{}
\begin{document}
 \maketitle

\begin{abstract}
Let $G$ be an $n$-vertex graph and $Q(G)$ be its signless Laplacian matrix. The $Q$-walk matrix of $G$, denoted by $W_Q(G)$, is $[e,Q(G)e,\ldots,Q^{n-1}(G)e]$, where $e$ is the all-one vector.  Let $G\circ P_m$ be the graph obtained from $G$ and $n$ copies of the path $P_m$  by identifying the $i$-th vertex of $G$ with an endvertex of the $i$-th copy of $P_m$ for each $i$. We prove that, 
 $$\det W_Q(G\circ P_m)=\pm (\det Q(G))^{m-1}(\det W_Q(G))^m$$
 holds for any $m\ge 2$. This gives a signless Laplacian counterpart of the following recently established identity \cite{wym2023}:
   $$\det W_A(G\circ P_m)=\pm (\det A(G))^{\lfloor\frac{m}{2}\rfloor}(\det W_A(G))^m,$$
      where $A(G)$ is the adjacency matrix of $G$ and $W_A(G)=[e,A(G)e,\ldots,A^{n-1}(G)e]$.
We also propose a conjecture  to unify the above two equalities.


\end{abstract}
\section{Introduction}
\label{intro}
Let $G$ be a simple graph with vertex set $\{1,\ldots,n\}$.  The \emph{adjacency matrix} of $G$ is the $n\times n$ symmetric matrix $A=(a_{i,j})$, where $a_{i,j}=1$ if $i$ and $j$ are adjacent;  $a_{i,j}=0$ otherwise.  For a graph $G$, the \emph{adjacency walk matrix} (or $A$-walk matrix) of $G$ is
\begin{equation*}
W_A=W_A(G):=[e,Ae,\ldots,A^{n-1}e],
\end{equation*}
where $e$ is the all-one vector.

Let $H$ be a rooted graph. The \emph{rooted product graph} \cite{godsil1978,schwenk1974} of $G$ and $H$, denoted by $G\circ H$, is a graph obtained from $G$ and $n$ copies of $H$ by identifying the  root vertex of the $i$-th copy of $H$ with  vertex $i$ of $G$ for $i=1,2,\ldots,n$. We are interested in the case that the rooted graph $H$ is the rooted path $P_m$ of order $m$ (taking an endvertex as the root vertex). For this special case, the authors established the following formula concerning  the determinant of  $W(G\circ P_m)$, which was conjectured in \cite{mao2022}.
\begin{theorem}[\cite{wym2023}]\label{twa}
	For any graph $G$ and integer $m\ge 2$,
\begin{equation}\label{wa}
\det W_A(G\circ P_m)=\pm (\det A(G))^{\lfloor\frac{m}{2}\rfloor}(\det W_A(G))^m.
\end{equation} 
\end{theorem}

The main aim of this paper is to give a signless Laplacian counterpart of the above formula. We recall that the \emph{signless Laplacian matrix} \cite{crs2007} of $G$, usually denoted by $Q(G)$ or simply $Q$, is $A(G)+D(G)$, where $A(G)$ is the adjacency matrix and $D(G)$ is the degree diagonal matrix. The \emph{signless Laplacian walk matrix} (or $Q$-walk matrix \cite{qjw2019}), denoted by $W_Q(G)$, is naturally defined by 
\begin{equation*}
W_Q(G)=[e,Qe,\ldots,Q^{n-1}e].
\end{equation*}
The main result of this paper is the following.
\begin{theorem}\label{twq}
		For any graph $G$ and integer $m\ge 2$,
 \begin{equation}\label{wq}
 \det W_Q(G\circ P_m)=\pm (\det Q(G))^{m-1}(\det W_Q(G))^m.
 \end{equation}
\end{theorem}
\begin{remark}\normalfont{
During the writing of this paper, the authors learned that for the special case that $m\in \{2,3\}$,  Theorem \ref{twq} was independently proved by Tian et al. \cite{twcs2023} using a different method.}
\end{remark}
Note that the power of the first factor in the RHS of Eq.~\eqref{wa} is $\lfloor\frac{m}{2}\rfloor$, whereas the corresponding number in Eq.~\eqref{wq} is $m-1$.  It would be desirable to give a general formula to unify Eqs.~\eqref{wa} and \eqref{wq}. The following definition is essentially the $A_\alpha$-matrix introduced by Nikiforov \cite{nikiforov}.  
\begin{definition}
 $A_\tau(G)=A(G)+\tau D(G)$ for $\tau\in \mathbb{R}$. 
\end{definition}
Note that $A_0(G)=A(G)$ and $A_1(G)=Q(G)$. We write 
\begin{equation*}
W_\tau(G)=[e,A_\tau e,\ldots,A_\tau^{n-1}e], 
\end{equation*}
where $A_\tau=A_\tau(G)$. We shall call $W_\tau(G)$ the $A_\tau$-walk matrix of $G$. Clearly, $W_\tau(G)=W_A(G)$ when $\tau=0$; and  $W_\tau(G)=W_Q(G)$ when $\tau=1$. We use $I_n$ to denote the identity matrix of order $n$.
\begin{conjecture}\label{mainconj}
	\begin{equation*}
	\det W_\tau(G\circ P_m)=\pm (\det A_\tau(G))^{\lfloor\frac{m}{2}\rfloor}(\det ((1-\tau^2)I_n+\tau A_\tau(G)))^{\lfloor\frac{m-1}{2}\rfloor}(\det W_\tau(G))^m.
	\end{equation*}
\end{conjecture}
One easily finds that Conjecture  \ref{mainconj} (if true) unifies Theorems \ref{twa} and \ref{twq}. The rest of this paper is devoted to proving Theorem \ref{twq}. The main strategy is similar to our previous paper \cite{wym2023}, which reduces the desired identity on the walk determinants (Eq.~\eqref{wa}) to some identities on resultants related to the second kind of Chebyshev polynomials. We find that, except for the computation of resultants of Chebyshev polynomials, most of the previous argument on the $A$-walk matrix in \cite{wym2023} can be easily extended to the $Q$-walk matrix or the more general $A_\tau$-walk matrix. 
\section{Eigenvalues and eigenvectors of $A_\tau(G\circ P_m)$}\label{comp_ev}

\begin{figure}\label{C4P3}
	\centering
	\includegraphics[height=4cm]{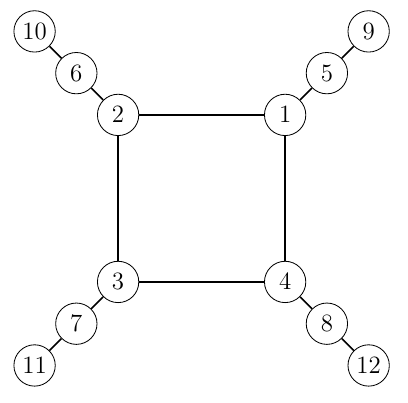}
	\caption{The rooted product $C_4\circ P_3$.}
\end{figure}
\begin{definition}\label{defz}
	$Z_0^{(\tau)}(x)=1$, $Z_1^{(\tau)}(x)=x-\tau$, and $Z_k^{(\tau)}(x)=(x-2\tau)Z_{k-1}^{(\tau)}(x)-Z_{k-2}^{(\tau)}(x)$ for $k\ge 2$.
\end{definition}
Let $B_1^{(\tau)}=\tau$ and \begin{equation*}
B_k^{(\tau)}=A_\tau(P_k)+\diag[\tau,0,0,\ldots,0]=\begin{pmatrix}
2\tau&1&&&\\1&2\tau&1&&\\&\ddots&\ddots&\ddots&\\&&1&2\tau&1\\&&&1&\tau
\end{pmatrix}, \text{~for~} k\ge 2.
\end{equation*}
Expanding $\det (xI_k-B_k^{(\tau)})$ by the first row, one easily finds that 
$$\det (xI_k-B_k^{(\tau)})=(x-2\tau)\det(xI_{k-1}-B_{k-1}^{(\tau)})-\det(xI_{k-2}-B_{k-2}^{(\tau)}),\text{~for~} k\ge 2.$$
This means that $\det (xI_k-B_k^{(\tau)})$ satisfies the same three-term recurrence relations as $Z_k^{(\tau)}(x)$.  Noting that $\det (xI_k-B_k^{(\tau)})=Z_k^{(\tau)}(x)$ for $k\in\{0,1\}$, the two sequences must be equal. Here  we make the convention that $B_0^{(\tau)}$ is a matrix of order 0 and $\det (xI_0-B_0^{(\tau)})=1$.
\begin{lemma}\label{zb}
	$Z_k^{(\tau)}(x)=\det(xI_k-B_k^{(\tau)})$ for all $k\ge 0$.
\end{lemma}

\begin{definition}\normalfont{
	Let $A=(a_{ij})$ be an $m\times n$ matrix and $B$  a $ p \times  q$ matrix. The \emph{Kronecker product} $A\otimes B$ is the  block matrix:
	$$A\otimes B=\begin{bmatrix}
	a_{11}B&\cdots&a_{1n}B\\
	\vdots&\ddots&\vdots\\
	a_{m1}B&\cdots&a_{mn}B
	\end{bmatrix}.
	$$
}
\end{definition}
By appropriately labeling the vertices in $G\circ P_m$ (see Figure \ref{C4P3} for an illustration), the  matrix $A_\tau(G\circ P_m)$ has a nice structure.
\begin{lemma}\label{adjGP}
	$A_\tau(G\circ P_m)=A_\tau(P_m)\otimes I_n+D_1\otimes A_\tau(G),$
	where $I_n$ is the identity matrix of order $n$ and $D_1$ is the diagonal matrix $\diag(1,0,\ldots,0)$ of order $m$.
	\end{lemma}
\begin{proof}
	We have
	\begin{eqnarray*}
	A_\tau(G\circ P_m)&=&A(G\circ P_m)+\tau D(G\circ P_m)\\
	&=&
	\begin{bmatrix}
		A(G)&I_n&&&\\
		I_n&0&I_n&&\\
		&I_n&\ddots&\ddots&\\
		&&\ddots&0&I_n\\
		&&&I_n&0
	\end{bmatrix}+
\tau\begin{bmatrix}
	D(G)+I_n&&&&\\
	&2I_n&&&\\
	&&\ddots&&\\
	&&&2I_n&\\
	&&&&I_n
\end{bmatrix}\\
&=&\begin{bmatrix}
	\tau I_n&I_n&&&\\
	I_n&2\tau I_n&I_n&&\\
	&I_n&\ddots&\ddots&\\
	&&\ddots&2\tau I_n& I_n\\
	&&&I_n&\tau I_n
\end{bmatrix}+
	\begin{bmatrix}
	A(G)+\tau D(G)&&&&\\
	&0&&&\\
	&&0&&\\
	&&&\ddots&\\
	&&&&0
\end{bmatrix}\\
&=&A_\tau(P_m)\otimes I_n+D_1\otimes A_\tau(G).
	\end{eqnarray*}
This completes the proof.
\end{proof}
For a graph $G$, we use $\phi_\tau(G;x)$ to denote $\det(xI-A_\tau(G))$ and we call it the $A_\tau$- characteristic polynomial of $G$. The roots of $\phi_\tau(G;x)=0$ are called the $A_\tau$-eigenvalues (spectrum) of $G$.
\begin{lemma}\label{chip}
	$\phi_\tau(P_k;x)=Z_k^{(\tau)}(x)+\tau Z_{k-1}^{(\tau)}(x)$ for $k\ge 1$.
	\end{lemma}
\begin{proof}
	We have $\phi_\tau(P_1;x)=x$ and $Z_1^{(\tau)}(x)+\tau Z_{0}^{(\tau)} (x)=(x-\tau)+\tau=x$. This means that Lemma \ref{chip} holds for $k=1$. Now, we assume $k\ge 2$. Then, we have

\begin{eqnarray}
\phi_\tau(P_k;x)&=&\begin{vmatrix}
x-\tau&-1&&&\\-1&x-2\tau&-1&&\\&\ddots&\ddots&\ddots&\\&&-1&x-2\tau&-1\\&&&-1&x-\tau\end{vmatrix}\nonumber\\
&=&\begin{vmatrix}
x-2\tau&-1&&&\\-1&x-2\tau&-1&&\\&\ddots&\ddots&\ddots&\\&&-1&x-2\tau&-1\\&&&-1&x-\tau
\end{vmatrix}+\begin{vmatrix}
\tau&0&&&\\-1&x-2\tau&-1&&\\&\ddots&\ddots&\ddots&\\&&-1&x-2\tau&-1\\&&&-1&x-\tau
\end{vmatrix}\nonumber\\
&=&\det(xI_k-B_k^{(\tau)})+\tau\det(xI_{k-1}-B_{k-1}^{(\tau)})\nonumber\\
&=&Z_k^{(\tau)}(x)+\tau Z_{k-1}^{(\tau)}(x)\nonumber.
\end{eqnarray}
This completes the proof. 
\end{proof}  
We need the following useful formula on determinants of block matrices; proofs can be found in,  e.g., \cite{b1989,ksw1999,s2000}. 
\begin{lemma}\label{detST} Let $S$ be a commutative subring of $T^{n×n}$, the set of all $n \times n$ matrices over a  commutative ring $T$, 
and let $M \in  S^{m×m}$. Then,
$$\det\nolimits_T M = {\det\nolimits_T}({\det\nolimits_S} M).$$
(the subscript indicates where the determinant is computed).
\end{lemma}
\begin{lemma}
	$\phi_\tau(G\circ P_m;x)=\prod\limits_{i=1}^n(Z_m^{(\tau)}(x)+(\tau-\lambda_i)Z_{m-1}^{(\tau)}(x)),$
	where $\lambda_1^{(\tau)},\ldots,\lambda_n^{(\tau)}$ are $A_\tau$-eigenvalues of $G$.
\end{lemma}
\begin{proof}
Letting $A_\tau=A_\tau(G)$ and noting that 
$$A_\tau(G\circ P_m)=\begin{pmatrix}
A_\tau+\tau I&I&&&\\
I&2\tau I&I&&\\
&\ddots&\ddots&\ddots&\\
&&I&2\tau I&I\\
&&&I&\tau I
\end{pmatrix},
$$
we have, by Lemma \ref{detST},
\begin{eqnarray}
\phi_\tau(G\circ P_m;x)&=&\begin{vmatrix}
(x-\tau)I-A_\tau&-I&&&\\-I&(x-2\tau)I&-I&&\\&\ddots&\ddots&\ddots&\\&&-I&(x-2\tau)I&-I\\&&&-I&(x-\tau)I\end{vmatrix}\nonumber\\
&=&\det\nolimits_{\mathbb{R}[x]}\left(\det\nolimits_{\langle A_\tau,I\rangle} \begin{pmatrix}
(x-\tau)I-A_\tau&-I&&&\\-I&(x-2\tau)I&-I&&\\&\ddots&\ddots&\ddots&\\&&-I&(x-2\tau)I&-I\\&&&-I&(x-\tau)I
\end{pmatrix}\right)\nonumber,
\end{eqnarray}
where $\langle A_\tau, I\rangle$ is the subring of the matrix ring $\mathbb{R}[x]^{n\times n}$ generated by  $\{A_\tau,I\}$  and is  clearly a commutative subring. Let $S=\langle A_\tau,I\rangle$. Then 
\begin{eqnarray}
&&\det\nolimits_{S} \begin{pmatrix}
(x-\tau)I-A_\tau&-I&&&\\-I&(x-2\tau)I&-I&&\\&\ddots&\ddots&\ddots&\\&&-I&(x-2\tau)I&-I\\&&&-I&(x-\tau)I
\end{pmatrix}\nonumber\\
&=&\det\nolimits_{S} \begin{pmatrix}
(x-\tau)I&-I&&&\\-I&(x-2\tau)I&\ddots&\\&\ddots&\ddots&-I\\&&-I&(x-\tau)I
\end{pmatrix}-\det\nolimits_{S} \begin{pmatrix}
A_\tau&0&&\\-I&(x-2\tau)I&\ddots&\\&\ddots&\ddots&-I\\&&-I&(x-\tau)I\end{pmatrix}\nonumber
\end{eqnarray}
\begin{eqnarray}
&=& \begin{vmatrix}
(x-\tau)&-1&&&\\-1&(x-2\tau)&\ddots&\\&\ddots&\ddots&-1\\&&-1&(x-\tau)
\end{vmatrix}\cdot I- \begin{vmatrix}
1&0&&\\-1&(x-2\tau)&\ddots&\\&\ddots&\ddots&-1\\&&-1&(x-\tau)\nonumber
\end{vmatrix}\cdot A_\tau\\
&=&\phi_\tau(P_m;x)\cdot I-(\det(xI_{m-1}-B_{m-1}^{(\tau)}))\cdot A_\tau\nonumber\\
&=&(Z_m^{(\tau)}(x)+\tau Z_{m-1}^{(\tau)} (x))\cdot I- Z_{m-1}^{(\tau)} (x)\cdot A_\tau,\nonumber
\end{eqnarray}
where the last equality uses Lemmas \ref{zb} and \ref{chip}.
It follows that 
\begin{eqnarray}
\phi_\tau(G\circ P_m)&=&\det( (Z_m^{(\tau)}(x)+\tau Z_{m-1}^{(\tau)} (x))\cdot I- Z_{m-1}^{(\tau)} (x)\cdot A_\tau)\nonumber\\
&=& \left(Z_{m-1}^{(\tau)}(x)\right)^n\det \left(\frac{Z_m^{(\tau)}(x)+\tau Z_{m-1}^{(\tau)} (x)}{Z_{m-1}^{(\tau)}(x)}\cdot I-A_\tau\right)\nonumber\\
&=& \left(Z_{m-1}^{(\tau)}(x)\right)^n\phi_\tau\left(G;\frac{Z_m^{(\tau)}(x)+\tau Z_{m-1}^{(\tau)} (x)}{Z_{m-1}^{(\tau)}(x)}\right)\nonumber\\
&=&\left(Z_{m-1}^{(\tau)}(x)\right)^n\prod_{i=1}^{n}\left(\frac{Z_m^{(\tau)}(x)+\tau Z_{m-1}^{(\tau)} (x)}{Z_{m-1}^{(\tau)}(x)}-\lambda_i^{(\tau)}\right)\nonumber\\
&=&\prod_{i=1}^{n}\left(Z_m^{(\tau)}(x)+(\tau-\lambda_i^{(\tau)}) Z_{m-1}^{(\tau)}(x)\right),\nonumber
\end{eqnarray}
completing the proof.
\end{proof}
To simplify the notations, we shall omit the superscript $\tau$ in the following argument. For example, the notation $Z_m^{(\tau)}(x)$ will be simplified as $Z_m(x)$. 
\begin{definition}\label{eigmu}\normalfont{
Let $\lambda_1,\ldots,\lambda_n$ denote the $A_\tau$-eigenvalues of $G$ and $\xi_1,\ldots,\xi_n$ be the corresponding normalized eigenvector. 	We use $\mu_i^{(j)}(j\in\{1,2,\ldots,m\})$ to denote all zeroes of $Z_m(x)+(\tau-\lambda_i)Z_{m-1}(x)$ for  $i\in\{1,2,\ldots,n\}$ and write $$\eta_i^{(j)}=\frac{1}{Z_{m-1}(\mu_i^{(j)})}\begin{bmatrix}
Z_{m-1}(\mu_i^{(j)})\\
Z_{m-2}(\mu_i^{(j)})\\
\vdots\\
Z_0(\mu_i^{(j)})
\end{bmatrix}\otimes \xi_i.$$
}
\end{definition}
It should be pointed out that $Z_{m-1}(\mu_i^{(j)})$ is never zero, see Corollary \ref{pzm} in Sect.~\ref{cdw}. The following lemma is a special case of \cite[Theorem 3.3.4]{szego}. 
\begin{lemma} \cite{szego}\label{distmu}
For any $i\in\{1,2,\ldots,n\}$, the polynomial $Z_m(x)+(\tau-\lambda_i)Z_{m-1}(x)$ has $m$ distinct real zeros.
\end{lemma}
The main result of this section is the following
\begin{lemma}\label{eigA}
	Let $\tilde{A}_\tau$ denote the $A_\tau$-matrix of $G\circ P_m$. Then $\tilde{A}_\tau\eta_i^{(j)}=\mu_i^{(j)}\eta_i^{(j)}$ for $i\in\{1,2,\ldots,n\}$ and $j\in\{1,2,\ldots,m\}$.
\end{lemma}
\begin{proof}
We fix $i$ and $j$ and write $z_k=Z_k(\mu_i^{(j)})$ $(k=0,1,\ldots, m-1)$ for simplicity. By Lemma \ref{adjGP} and some basic properties of the Kronecker product, we obtain
	\begin{eqnarray}\label{Aeta}
\tilde{A}_\tau\eta_i^{(j)} &=&\frac{1}{z_{m-1}}(A_\tau(P_m)\otimes I_n+D_1\otimes A_\tau(G))((z_{m-1},z_{m-2},\ldots,z_0)^\T\otimes \xi_i)\nonumber\\
&=&\frac{1}{z_{m-1}}\left(\left(A_\tau(P_m)\begin{bmatrix}
z_{m-1}\\
z_{m-2}\\
\vdots\\
z_0
\end{bmatrix}\right)\otimes \xi_i+\begin{bmatrix}
z_{m-1}\\
0\\
\vdots\\
0
\end{bmatrix}\otimes (\lambda_i \xi_i)\right)\nonumber\\
&= &\frac{1}{z_{m-1}}\left(A_\tau(P_m)\begin{bmatrix}
z_{m-1}\\
z_{m-2}\\
\vdots\\
z_0
\end{bmatrix}+\begin{bmatrix}
\lambda_i z_{m-1}\\
0\\
\vdots\\
0
\end{bmatrix}\right)\otimes \xi_i\nonumber\\
&= &\frac{1}{z_{m-1}}\begin{bmatrix}
z_{m-2}+\tau z_{m-1}+\lambda_i z_{m-1}\\
z_{m-3}+2\tau z_{m-2}+z_{m-1}\\
\vdots\\
z_0+2\tau z_1+z_2\\
\tau z_0+z_1
\end{bmatrix}\otimes \xi_i.
\end{eqnarray}
By Definition \ref{eigmu}, we see that $\lambda_i z_{m-1}=z_m+\tau z_{m-1}$. Noting that, by Definition \ref{defz}, $\tau z_0+z_1=\tau+(\mu_i^{(j)}-\tau)=\mu_i^{(j)}$ and $z_k+2\tau z_{k+1}+z_{k+2}=\mu_i^{(j)}z_{k+1}$ for any $k\ge 0$, we obtain
\begin{equation*}
\begin{bmatrix}
z_{m-2}+\tau z_{m-1}+\lambda_i z_{m-1}\\
z_{m-3}+2\tau z_{m-2}+z_{m-1}\\
\vdots\\
z_0+2\tau z_1+z_2\\
\tau z_0+z_1
\end{bmatrix}=\begin{bmatrix}
z_{m-2}+2\tau z_{m-1}+z_{m}\\
z_{m-3}+2\tau z_{m-2}+z_{m-1}\\
\vdots\\
z_0+2\tau z_1+z_2\\
\mu_{i}^{(j)}
\end{bmatrix}=\mu_i^{(j)}\begin{bmatrix}
z_{m-1}\\
z_{m-2}\\
\vdots\\
z_1\\
z_0
\end{bmatrix}.
\end{equation*}
Now, Eq. \eqref{Aeta} becomes
$$\tilde{A}_\tau\eta_i^{(j)}=\frac{1}{z_{m-1}}\mu_{i}^{(j)}\begin{bmatrix}
z_{m-1}\\
z_{m-2}\\
\vdots\\
z_1\\
z_0
\end{bmatrix}\otimes\xi_i=\mu_i^{(j)}\eta_i^{(j)},$$
completing the proof.
\end{proof}
\begin{remark}\normalfont{
For the special case when $\tau=1$, Lemma \ref{eigA} was obtained in \cite{lhh2017} using a slightly different argument.}
\end{remark}
\section{Computing $\det W_\tau(G\circ P_m)$}\label{cdw}
\begin{definition}
	$Z(x)=\sum_{k=0}^{m-1}Z_k(x)$.
\end{definition}
The main aim of this section is to prove the following
\begin{proposition}\label{dwt}
	\begin{equation*}
	\det W_\tau(G\circ P_m)=\pm \left(\det W_\tau(G)\right)^m\prod_{i=1}^n\prod_{j=1}^m Z(\mu_i^{(j)}).
	\end{equation*}
\end{proposition}
We need a useful result due to Schur, see \cite[{\S6.71}]{szego}.
\begin{lemma}\cite{szego}\label{sch} Let $\{p_k(x)\}$ be a sequence of polynomials satisfying the three-term recurrence formula
	\begin{equation}\label{pk}
	p_k(x) = (a_kx + b_k) p_{k-1}(x)-c_k p_{k-2}(x), 	k = 2, 3, \ldots,
	\end{equation}
	and the intial conditions $p_0(x) = 1, p_1(x) = a_1x + b_1$. Assume that $a_1a_kc_k\neq 0$		for $k > 1$ and let $\{x_{j,k}\colon\,1\le j\le k\}$ be the zeros of $p_{k}(x)$. Then,
	\begin{equation}\label{del}
	\prod_{j=1}^{k}p_{k-1}(x_{j,k})=(-1)^\frac{k(k-1)}{2}\prod_{j=1}^{k}a_j^{k-2j+1}c_j^{j-1},
	\end{equation}
	where $c_1^0=1$ by convention.
		\end{lemma}
\begin{corollary}\label{pzm}
	For any $i\in \{1,2,\ldots,n\}$,
	\begin{equation}\label{delZ}
\prod_{j=1}^{m}Z_{m-1}(\mu_i^{(j)})=(-1)^\frac{m(m-1)}{2}.
	\end{equation}
\end{corollary}
\begin{proof}
	Define
	\begin{equation*}
	p_k(x)=\begin{cases}
	Z_k(x)&k=0,1,\ldots,m-1,\\
	Z_m(x)+(\tau-\lambda_i)Z_{m-1}(x)&k=m.
	\end{cases}
	\end{equation*}
	Then $p_k(x)$ clearly satisfies Eq.~\eqref{pk} with $a_1=1$ and $a_k=c_k=1$ for $k=2,\ldots,m$. We do not need to write the exact values of $b_k$'s since they do not appear in Eq.~\eqref{del}. Noting that $\{\mu_i^{(j)}\colon\, 1\le j\le m\}$ are the zeros of $p_m(x)$, Eq.~\eqref{delZ} clearly follows from Eq.~\eqref{del} for $k=m$. 
\end{proof}
\begin{corollary}\label{ppmu}
	\begin{equation*}
	\prod_{j_2=1}^{m}\prod_{j_1=1}^{m}\left(\mu_{i_2}^{(j_2)}-\mu_{i_1}^{(j_1)}\right)=(-1)^{\frac{m(m-1)}{2}}(\lambda_{i_2}-\lambda_{i_1})^m.
	\end{equation*}
\end{corollary}
\begin{proof}
Note that $Z_m(x)+(\tau-\lambda_i)Z_{m-1}(x)$ is monic and has roots $\mu_i^{(1)},\ldots,\mu_i^{(m)}$, we have
$Z_m(x)+(\tau-\lambda_i)Z_{m-1}(x)=\prod_{j_1=1}^m(x-\mu_i^{(j_1)}).$ Thus,
	\begin{equation*}
Z_m(\mu_{i_2}^{(j_2)})+(\tau-\lambda_{i_1})Z_{m-1}(\mu_{i_2}^{(j_2)})=\prod_{j_1=1}^{m}\left(\mu_{i_2}^{(j_2)}-\mu_{i_1}^{(j_1)}\right).
\end{equation*}
This, together with the fact that $Z_m(\mu_{i_2}^{(j_2)})+(\tau-\lambda_{i_2})Z_{m-1}(\mu_{i_2}^{(j_2)})=0$ implies
\begin{eqnarray}
	\prod_{j_2=1}^{m}\prod_{j_1=1}^{m}\left(\mu_{i_2}^{(j_2)}-\mu_{i_1}^{(j_1)}\right)&=&\prod_{j_2=1}^{m}\left(Z_m(\mu_{i_2}^{(j_2)})+(\tau-\lambda_{i_1})Z_{m-1}(\mu_{i_2}^{(j_2)})\right)\nonumber\\
&=&\prod_{j_2=1}^{m}\left(-(\tau-\lambda_{i_2})Z_{m-1}(\mu_{i_2}^{(j_2)})+(\tau-\lambda_{i_1})Z_{m-1}(\mu_{i_2}^{(j_2)})\right)\nonumber\\
&=&(\lambda_{i_2}-\lambda_{i_1})^m\prod_{j_2=1}^{m}Z_{m-1}(\mu_{i_2}^{(j_2)})\nonumber\\
&=&(-1)^{\frac{m(m-1)}{2}}(\lambda_{i_2}-\lambda_{i_1})^m\nonumber,
\end{eqnarray}
where we use Corollary \ref{pzm} in the last equality. This completes the proof.
\end{proof}
In \cite{mao2015}, Mao et al. obtained an explicit formula to compute $\det W_A(G)$   using the adjacency eigenvalues and eigenvectors of $G$. The method can be naturally extended  to compute $\det W_\tau(G)$ for any $\tau\in \mathbb{R}$.
\begin{lemma}[\cite{mao2015}]\label{basicW}
	Let $\lambda_i$ be the $A_\tau$-eigenvalues of $G$ with  eigenvector $\xi_i$ for $i=1,2,\ldots,n$. Then
	$$\det W_\tau(G)= \frac{\prod_{1\le i_1< i_2\le n}(\lambda_{i_2}-\lambda_{i_1})\prod_{1\le i\le n}(e_n^\T \xi_i)}{\det[\xi_1,\xi_2,\ldots,\xi_n]}.$$
\end{lemma}
Let $\Omega=\{(i,j)\colon\,1\le i\le n \text{~and~} 1\le j\le m\}$ with the co-lexicographical order: $(i_1,j_1)<(i_2,j_2) $ if either $j_1<j_2$, or $j_1=j_2$ and $i_1<i_2$. The following formula of $\det W_\tau(G\circ P_m)$ is an immediate consequence of Lemmas \ref{basicW} and Lemma \ref{eigA}.
\begin{corollary}\label{dwt3}
\begin{equation*}
	\det W_\tau(G\circ P_m)= \frac{\prod_{(i_1,j_1)<(i_2,j_2)}(\mu_{i_2}^{(j_2)}-\mu_{i_1}^{(j_1)})\prod_{(i,j)\in \Omega}(e_{mn}^\T \eta_i^{(j)})}{\det[\eta_1^{(1)},\ldots,\eta_n^{(1)};\ldots;\eta_1^{(m)},\ldots,\eta_n^{(m)}]}.
\end{equation*}
\end{corollary}
\begin{lemma}\label{exm}
	$\prod_{(i,j)\in \Omega}(e_{mn}^\T \eta_i^{(j)})=(-1)^\frac{m(m-1)n}{2} \left(\prod_{(i,j)\in \Omega}Z(\mu_i^{(j)})\right)\left(\prod_{1\le i\le n}e_n^\T \xi_i\right)^m.$
\end{lemma}
\begin{proof}
	Noting $e_{mn}^\T=(e_m)^\T \otimes (e_n)^\T$ together with the definitions of $\eta_i^{(j)}$ and $Z(x)$, we have
	\begin{eqnarray}
	e_{mn}^\T\eta_i^{(j)}&=&(e_m^\T\otimes e_n^\T)\left(\frac{1}{Z_{m-1}(\mu_i^{(j)})}\begin{bmatrix}
	Z_{m-1}(\mu_i^{(j)})\\
	Z_{m-2}(\mu_i^{(j)})\\
	\vdots\\
	Z_0(\mu_i^{(j)})
	\end{bmatrix}\otimes \xi_i\right)\nonumber\\
	&=&\frac{1}{Z_{m-1}(\mu_i^{(j)})}\left(\sum_{k=0}^{m-1}Z_k(\mu_i^{(j)})\right)e_n^\T\xi_i\nonumber\\
	&=&\frac{1}{Z_{m-1}(\mu_i^{(j)})}Z(\mu_i^{(j)})e_n^\T\xi_i.\nonumber
	\end{eqnarray}
By Corollary \ref{pzm}, we have $\prod_{(i,j)\in \Omega} Z_{m-1}(\mu_i^{(j)})=\prod_{i=1}^n\prod_{j=1}^mZ_{m-1}(\mu_i^{(j)})=(-1)^\frac{m(m-1)n}{2}$ and hence Lemma \ref{exm} holds.
\end{proof}
\begin{lemma}
	\begin{equation*}
\prod_{(i_1,j_1)<(i_2,j_2)}(\mu_{i_2}^{(j_2)}-\mu_{i_1}^{(j_1)})=\left(\prod_{i=1}^{n}\prod_{1\le j_1<j_2\le m}\left(\mu_i^{(j_2)}-\mu_i^{(j_1)}\right)\right)\left(\prod_{1\le i_1< i_2\le n}(\lambda_{i_2}-\lambda_{i_1})\right)^m.
	\end{equation*}
\end{lemma}
\begin{proof}
	We have 
	\begin{equation}\label{pmumu}
\prod_{(i_1,j_1)<(i_2,j_2)}(\mu_{i_2}^{(j_2)}-\mu_{i_1}^{(j_1)})=\left(\prod_{i=1}^{n}\prod_{1\le j_1<j_2\le m}\left(\mu_{i}^{(j_2)}-\mu_{i}^{(j_1)}\right)\right)\left(\prod_{i_1\neq i_2}\prod_{(i_1,j_1)<(i_2,j_2)}\left(\mu_{i_2}^{(j_2)}-\mu_{i_1}^{(j_1)}\right)\right).
	\end{equation}
The second factor can be regrouped as 
	\begin{eqnarray}\label{sf}
	&&\prod_{1\le i_1< i_2\le n}\left(\prod_{(i_1,j_1)<(i_2,j_2)}\left(\mu_{i_2}^{(j_2)}-\mu_{i_1}^{(j_1)}\right)\right)\left(\prod_{(i_2,j_2)<(i_1,j_1)}\left(\mu_{i_1}^{(j_1)}-\mu_{i_2}^{(j_2)}\right)\right)\nonumber\\
	&=&\prod_{1\le i_1< i_2\le n}\left(\prod_{j_2=1}^{m}\prod_{j_1=1}^{m}\left(\mu_{i_2}^{(j_2)}-\mu_{i_1}^{(j_1)}\right)\right)\left(\prod_{(i_2,j_2)<(i_1,j_1)}(-1)\right).
	\end{eqnarray}
	By Corollary \ref{ppmu}, we have $\prod_{j_2=1}^{m}\prod_{j_1=1}^{m}\left(\mu_{i_2}^{(j_2)}-\mu_{i_1}^{(j_1)}\right)=(-1)^{\frac{m(m-1)}{2}}(\lambda_{i_2}-\lambda_{i_1})^m$. Moreover, noting that for any $i_1$ and $i_2$ with $1\le i_1<i_2\le n$, the inequality $(i_2,j_2)<(i_1,j_1)$ (in the co-lexicographical order) if and only $j_2<j_1$, we obtain  $\prod_{(i_2,j_2)<(i_1,j_1)}(-1)=(-1)^{\frac{m(m-1)}{2}}$. Thus Eq. \eqref{sf} reduces to 
	\begin{equation*}
	\prod_{1\le i_1< i_2\le n} (\lambda_{i_2}-\lambda_{i_1})^m.
	\end{equation*}
	and the lemma follows by Eq.~\eqref{pmumu}.
\end{proof}
\begin{lemma}\label{deteta}
	\begin{equation*}
	\det[\eta_1^{(1)},\ldots,\eta_n^{(1)};\ldots;\eta_1^{(m)},\ldots,\eta_n^{(m)}]= \left(\det[\xi_1,\xi_2,\ldots,\xi_n]\right)^m\prod_{i=1}^{n}\prod_{1\le j_1< j_2\le m}\left(\mu_i^{(j_2)}-\mu_i^{(j_1)}\right).
	\end{equation*}
\end{lemma}
\begin{proof}
	Let $E^{(j)}=[\eta_1^{(j)},\eta_2^{(j)},\ldots,\eta_n^{(j)}]$ and  $z_k^{(i,j)}=\frac{Z_k(\mu_i^{(j)})}{Z_{m-1}(\mu_i^{(j)})}$ for $i\in\{1,2\,\ldots,n\}$, $j\in\{1,2,\ldots,m\}$  and $k\in \{0,1,\ldots,m-1\}$. We have
	\begin{eqnarray}\label{ej}
	E^{(j)}&=&\left[\begin{bmatrix}
z_{m-1}^{(1,j)}\cdot \xi_1\\
z_{m-2}^{(1,j)} \cdot\xi_1\\
\vdots\\
z_{0}^{(1,j)}\cdot \xi_1
	\end{bmatrix},\begin{bmatrix}
	z_{m-1}^{(2,j)} \cdot\xi_2\\
	z_{m-2}^{(2,j)}\cdot\xi_2\\
	\vdots\\
	z_{0}^{(2,j)} \cdot\xi_2
	\end{bmatrix},\cdots,\begin{bmatrix}
	z_{m-1}^{(n,j)} \cdot\xi_n\\
	z_{m-2}^{(n,j)} \cdot\xi_n\\
	\vdots\\
	z_{0}^{(n,j)} \cdot\xi_n
	\end{bmatrix}\right]\nonumber\\
	&=&\begin{bmatrix}
	[\xi_1,\xi_2\,\ldots,\xi_n]\cdot\diag[z_{m-1}^{(1,j)},	z_{m-1}^{(2,j)},\ldots,	z_{m-1}^{(n,j)}]\\
	[\xi_1,\xi_2\,\ldots,\xi_n]\cdot\diag[z_{m-2}^{(1,j)},	z_{m-2}^{(2,j)},\ldots,	z_{m-2}^{(n,j)}]\\
	\vdots\\
	[\xi_1,\xi_2\,\ldots,\xi_n]\cdot\diag[z_{0}^{(1,j)},	z_{0}^{(2,j)},\ldots,	z_{0}^{(n,j)}]
	\end{bmatrix}.
	\end{eqnarray}
	Let $Q=[\xi_1,\xi_2,\ldots,\xi_n]$ and $K_{j_1,j_2}=\diag[z_{m-j_1}^{(1,j_2)},	z_{m-j_1}^{(2,j_2)},\ldots,	z_{m-j_1}^{(n,j_2)}]$ for $j_1,j_2\in\{1,2\ldots,m\}$. Then we can rewrite Eq. \eqref{ej} as 
	\begin{equation*}
	E^{(j)}=\begin{bmatrix} Q\cdot K_{1,j}\\
	 Q\cdot K_{2,j}\\
	 \vdots\\
	  Q\cdot K_{m,j}
	  \end{bmatrix}=\begin{bmatrix}Q&&&\\
	  &Q&&\\
	  &&\ddots&\\
	  &&&Q\end{bmatrix}\begin{bmatrix}  K_{1,j}\\
	   K_{2,j}\\
	  \vdots\\
	   K_{m,j}
	  \end{bmatrix}.
	\end{equation*}
which implies the following identity:
\begin{equation}\label{e1n}
[E^{(1)},E^{(2)},\ldots,E^{(n)}]=\begin{bmatrix}
Q&&&\\
&Q&&\\
&&\ddots&\\
&&&Q
\end{bmatrix} \begin{bmatrix}
K_{1,1}&K_{1,2}&\cdots&K_{1,m}\\
K_{2,1}&K_{2,2}&\cdots&K_{2,m}\\
\cdots&\cdots&\cdots&\cdots\\
K_{m,1}&K_{m,2}&\cdots&K_{m,m}
\end{bmatrix}.
\end{equation}
Since each $K_{j_1,j_2}$ is a diagonal matrix, one easily sees that the block matrix $(K_{j_1,j_2})$ is permutation similar to the following block diagonal matrix
\begin{equation*}
\begin{bmatrix}
L_1&&&\\
&L_2&&\\
&&\ddots&\\
&&&L_n
\end{bmatrix},
\end{equation*}
where each $L_i$ is a square matrix of order $m$ and the $(j_1,j_2)$-th entry of  $L_i$ is the $i$-th diagonal entry of $K_{j_1,j_2}$.

We have
 \begin{eqnarray}\label{Li}
L_i&=&\begin{bmatrix}
z_{m-1}^{(i,1)}& z_{m-1}^{(i,2)}&\ldots& z_{m-1}^{(i,m)}\\
z_{m-2}^{(i,1)}& z_{m-2}^{(i,2)}&\ldots& z_{m-2}^{(i,m)}\\
\cdots&\cdots&\cdots&\cdots\\
z_{0}^{(i,1)}& z_{0}^{(i,2)}&\ldots& z_{0}^{(i,m)}\\
\end{bmatrix}\nonumber\\
&=&\begin{bmatrix}
Z_{m-1}(\mu_i^{(1)})&\ldots&  Z_{m-1}(\mu_i^{(m)})\\
Z_{m-2}(\mu_i^{(1)})&\ldots&  Z_{m-2}(\mu_i^{(m)})\\
\cdots&\cdots&\cdots\\
Z_{0}(\mu_i^{(1)})&\ldots&  Z_{0}(\mu_i^{(m)})\\
\end{bmatrix}\begin{bmatrix}
\frac{1}{Z_{m-1}(\mu_i^{(1)})}&&\\
&\ddots&\\
&& \frac{1}{Z_{m-1}(\mu_i^{(m)})}\\
\end{bmatrix}.
\end{eqnarray}
Note that $Z_k(x)$ is a monic polynomial with degree $k$ for each nonnegative $k$. By some evident row operations, one easily sees that the determinant of the first factor in Eq.~\eqref{Li} equals
 \begin{equation*}
 \det\begin{bmatrix}
 (\mu_i^{(1)})^{m-1}&(\mu_i^{(2)})^{m-1}&\cdots& (\mu_i^{(m)})^{m-1}\\
 \cdots&\cdots&\cdots&\cdots\\
 \mu_i^{(1)}&\mu_i^{(2)}&\cdots& \mu_i^{(m)}\\
 1&1& \cdots&1
 \end{bmatrix}=(-1)^{\frac{m(m-1)}{2}}\prod_{1\le j_1< j_2\le m}\left(\mu_i^{(j_2)}-\mu_i^{(j_1)}\right).
 \end{equation*}
 It follows from Eq.~\eqref{Li} and Corollary \ref{pzm} that 
 $\det L_i=\prod_{1\le j_1\le j_2}(\mu_i^{(j_2)}-\mu_i^{(j_1)})$. Finally, by Eq.~\eqref{e1n}, we obtain
 $$\det [E^{(1)},E^{(2)},\ldots,E^{m}]=(\det Q)^m\prod_{1\le i\le n}\det(L_i)=\left(\det Q\right)^m\prod_{i=1}^{n}\prod_{1\le j_1< j_2\le m}\left(\mu_i^{(j_2)}-\mu_i^{(j_1)}\right).$$
 This completes the proof.
\end{proof}
\noindent\textbf{Proof of Proposition \ref{dwt}} By Corollary \ref{dwt3} and Lemmas \ref{exm}-\ref{deteta}, we have
\begin{eqnarray*}
&&\det W_\tau(G\circ P_m)\\
&=& \frac{\prod_{(i_1,j_1)<(i_2,j_2)}(\mu_{i_2}^{(j_2)}-\mu_{i_1}^{(j_1)})\prod_{(i,j)\in \Omega}(e_{mn}^\T \eta_i^{(j)})}{\det[\eta_1^{(1)},\ldots,\eta_n^{(1)};\ldots;\eta_1^{(m)},\ldots,\eta_n^{(m)}]}\\
&=&\frac{\left(\prod_{1\le i_1< i_2\le n}(\lambda_{i_2}-\lambda_{i_1})\right)^m (-1)^\frac{m(m-1)n}{2} \left(\prod_{(i,j)\in \Omega}Z(\mu_i^{(j)})\right)\left(\prod_{1\le i\le n}e_n^\T \xi_i\right)^m}{\left(\det[\xi_1,\xi_2,\ldots,\xi_n]\right)^m}\\
&=&(-1)^\frac{m(m-1)n}{2}\left(\prod_{(i,j)\in \Omega}Z(\mu_i^{(j)})\right)\left(\det W_\tau(G)\right)^m.
\end{eqnarray*}
This completes the proof of Proposition \ref{dwt}.

\section{Proof of Theorem \ref{twq}}\label{pft}
According to Proposition \ref{dwt}, to show Theorem \ref{twq}, we need to determine the exact value of $\prod_{i=1}^n\prod_{j=1}^m Z^{(1)}(\mu_i^{(j)})$. We reserve the superscript $\tau$ in this section.

 Let $\{U_n(x)\}$ and $\{W_n(x)\}$ be the Chebyshev polynomials of the second and the fourth kinds \cite{mason2003}, defined by the same three-term recurrence  $X_n(x)=2x X_{n-1}(x)-X(x)$ but with different initial values:
 \begin{equation*}
 \begin{cases}
 U_0(x)=1\\
 U_1(x)=2x
 \end{cases} \text{and~}
 \begin{cases}
 W_0(x)=1\\
 W_1(x)=2x+1.
 \end{cases}
 \end{equation*} 
 The following simple observation is vital.
 \begin{observation}
 $Z_n^{(0)}(x)=U_n(\frac{x}{2})$ and 	$Z_n^{(1)}(x)=W_n(\frac{x}{2}-1)$.
 \end{observation}
 \begin{definition}\normalfont{
 		Let $f(x)=a_nx^n+a_{n-1}x^{n-1}+\cdots+a_1x+a_0$ and $g(x)=b_mx^m+b_{m-1}x^{m-1}+\cdots+b_1x+b_0$. The resultant of $f(x)$ and $g(x)$, denoted by  $\Res(f(x),g(x))$, is defined to be
 		$$a_n^mb_m^n\prod_{1\le i\le n,1\le j\le m}(\alpha_i-\beta_j),$$
 		where $\alpha_i$'s and $\beta_j$'s are the roots (in complex field $\mathbb{C}$) of $f(x)$ and $g(x)$, respectively.
 	}
 \end{definition}
 We list some basic properties of resultants for convenience.
 \begin{lemma}\label{basicres}
 	Let $f(x)=a_nx^n+\cdots+a_0=a_n\prod_{i=1}^n(x-\alpha_i)$ and $g(x)=b_mx^m+\cdots+b_0=b_m\prod_{j=1}^m(x-\beta_j)$. Then the followings hold:\\
 	\textup{(\rmnum{1})} $\Res(f(x),g(x))=a_n^m\prod_{i=1}^{n}g(\alpha_i) =(-1)^{mn}b_m^{n}\prod_{j=1}^m f(\beta_j);$\\
 	\textup{(\rmnum{2})} $\Res(f(tx+s),g(tx+s))=\Res(f(tx),g(tx))=t^{mn}\Res(f(x),g(x))$ for any $t\in \mathbb{C}\setminus\{0\}$ and $s\in \mathbb{C}$.
 \end{lemma}
 \begin{lemma} \label{resw}
 	For any positive integer $m$ and complex number $t$,
 $$	\Res\left(W_m(x)+tW_{m-1}(x),\sum_{k=0}^{m-1}W_k(x)\right)=(-1)^{m(m-1)}2^{m(m-1)}(1-t)^{m-1}.$$
 \end{lemma}
\begin{proof}
 As $W_k(\cos \theta)=\frac{\sin(k+\frac{1}{2})\theta}{\sin\frac{\theta}{2}}$ \cite{mason2003} and using the identity 
 $$\sin \left(k+\frac{1}{2}\right)\theta=\frac{\cos k\theta-\cos(k+1)\theta}{2\sin\frac{\theta}{2}},$$ we find that 
 	\begin{equation}\label{sw}
 	\sum_{k=0}^{m-1}W_k(\cos \theta)=\frac{1-\cos m\theta}{2\sin^2\frac{\theta}{2}}.
 	\end{equation}
 	
 	We claim that the zeros of $\sum_{k=0}^{m-1}W_k(x)$ (including multiplicity) are exactly $\beta_j:=\cos\frac{2\pi j}{m}$ for $1\le j\le m-1$. Let $\theta_i=\frac{2\pi i}{m}$ for $1\le i\le \lfloor\frac{m}{2}\rfloor$. Assume $m$ is even. Clearly,   by Eq.~\eqref{sw}, each  $\cos \theta_i$ with $1\le i\le \frac{m}{2}$ is a  root of $\sum_{k=0}^{m-1}W_k(x)$. Moreover, except for $i=\frac{m}{2}$, each $\cos \theta_i$ is a multiple root of  $\sum_{k=0}^{m-1}W_k(x)$ as $\theta_i$ is a multiple root of $\frac{1-\cos m\theta}{2\sin ^2\frac{\theta}{2}}$ and $\frac{\mathrm{d} \cos \theta}{\mathrm{d}  \theta}|_{\theta=\theta_i}=-\sin \theta_i\neq 0$. Noting that $\sum_{k=0}^{m-1}W_k(x)$ is a polynomial of degree $m-1$, it has exactly $\frac{m}{2}$ double roots $\cos\theta_i$ for $1\le i<\frac{m}{2}$ together with a single root $\cos \theta_\frac{m}{2}$, which is $-1$. For the case that $m$ is odd, a similar argument shows that $\sum_{k=0}^{m-1}W_k(x)$ has exactly $\frac{m-1}{2}$ double roots $\cos \theta_i$ for $1\le i\le \frac{m-1}{2}$.  This proves the  claim.
 	
 	By the claim and noting that the leading term of $\sum_{k=0}^{m-1}W_k(x)$ is $2^{m-1}x^{m}$, we have 
 	\begin{eqnarray*}
 	&&	\Res\left(W_m(x)+tW_{m-1}(x),\sum_{k=0}^{m-1}W_k(x)\right)\\
 	&=&(-1)^{m(m-1)}2^{(m-1)m}\prod_{j=1}^{m-1}(W_m(\beta_j)+tW_{m-1}(\beta_j))\\
 	 	&=&(-1)^{m(m-1)}2^{m(m-1)}\prod_{j=1}^{m-1}\left(\frac{\sin (m+\frac{1}{2})\frac{2\pi j}{m}}{\sin\frac{\pi j}{m}}+t\frac{\sin (m-\frac{1}{2})\frac{2\pi j}{m}}{\sin\frac{\pi j}{m}}\right)\\
 	 		&=&(-1)^{m(m-1)}2^{m(m-1)}(1-t)^{m-1}.
 	\end{eqnarray*}
 This completes the proof.
\end{proof}
\begin{remark}\normalfont{
A similar version of Lemma \ref{resw} concerning the second kind of Chebyshev polynomials was reported in our previous paper \cite{wym2023}. The current proof is a slight modification of the original proof due to Terrence Tao \cite{MO}.}
\end{remark}
\begin{proposition}\label{ppz}
	$\prod_{i=1}^n\prod_{j=1}^m Z^{(1)}(\mu_i^{(j)})=(-1)^{m(m-1)n}(\det Q(G))^{m-1}.$
\end{proposition}
\begin{proof}
Recall that $Z^{(1)}(x)=\sum_{k=0}^{m-1}Z_k^{(1)}(x)$ and $\mu_i^{(j)}$'s are roots of $Z_m^{(1)}(x)+(1-\lambda_i)Z_{m-1}^{(1)}(x)$, where $\lambda_i$'s are $Q$-eigenvalues of $G$. Since $Z_m^{(1)}(x)+(1-\lambda_i)Z_{m-1}^{(1)}(x)$ is monic and $Z_k^{(1)}(x)=W_k(\frac{x}{2}-1)$, Lemma  \ref{basicres} implies
\begin{eqnarray*}
\prod_{j=1}^{m}Z^{(1)}(\mu_i^{(j)})&=&\Res\left(Z_m^{(1)}(x)+(1-\lambda_i)Z_{m-1}^{(1)}(x),\sum_{k=0}^{m-1}Z_k^{(1)}(x)\right)\\
&=&\Res\left(W_m\left(\frac{x}{2}-1\right)+(1-\lambda_i)W_{m-1}\left(\frac{x}{2}-1\right),\sum_{k=0}^{m-1}W_k\left(\frac{x}{2}-1\right)\right)\\
&=&\left(\frac{1}{2}\right)^{m(m-1)}\Res\left(W_m(x)+(1-\lambda_i)W_{m-1}(x)),\sum_{k=0}^{m-1}W_k(x)\right)\\
&=&(-1)^{m(m-1)}\lambda_i^{m-1}.
\end{eqnarray*}
Noting that $\prod_{i=1}^{n}\lambda_i=\det Q(G)$, we obtain
	$$\prod_{i=1}^n\prod_{j=1}^m Z^{(1)}(\mu_i^{(j)})=\prod_{i=1}^{n}(-1)^{m(m-1)}\lambda_i^{m-1}=(-1)^{m(m-1)n}(\det Q(G))^{m-1},$$
	completing the proof.
\end{proof}
\noindent\textbf{Proof of Theorem \ref{twq}} By Propositions \ref{dwt} and \ref{ppz}, Theorem \ref{twq} follows.

We end this paper by suggesting a possible way to resolve Conjecture \ref{mainconj}.
\begin{proposition}
	Let $\{Z_k^{(\tau)}(x)\}$ be the sequence of polynomials described in Definition \ref{defz}.  If 
	\begin{equation}\label{conres}
\Res\left(Z_m^{(\tau)}(x)+tZ_{m-1}^{(\tau)}(x),\sum_{k=0}^{m-1}Z_k^{(\tau)}(x)\right)=(-1)^{m(m-1)}(\tau-t)^{\lfloor\frac{m}{2}\rfloor}(1-\tau t)^{\lfloor\frac{m-1}{2}\rfloor}
	\end{equation}
	then Conjecture \ref{mainconj} holds.
\end{proposition}
\begin{proof}
	By Proposition \ref{dwt}, it suffices to show
	  \begin{equation*}
\prod_{i=1}^n\prod_{j=1}^m Z^{(\tau)}(\mu_i^{(j)})=\pm (\det A_\tau(G))^{\lfloor\frac{m}{2}\rfloor}(\det ((1-\tau^2)I_n+\tau A_\tau(G)))^{\lfloor\frac{m-1}{2}\rfloor}.
	\end{equation*}
By a similar argument as in Proposition \ref{ppz} and using Eq. \eqref{conres}, we obtain
\begin{eqnarray}\label{pzt}
\prod_{j=1}^m Z^{(\tau)}(\mu_i^{(j)})&=&\Res\left(Z_m^{(\tau)}(x)+(\tau-\lambda_i)Z_{m-1}^{(\tau)}(x),\sum_{k=0}^{m-1}Z_{k}^{(\tau)}(x)\right)\nonumber\\
&=&(-1)^{m(m-1)} (\tau-(\tau-\lambda_i))^{\lfloor\frac{m}{2}\rfloor}(1-\tau(\tau-\lambda_i))^{\lfloor\frac{m-1}{2}\rfloor}\nonumber\\
&=&\pm \lambda_i^{\lfloor\frac{m}{2}\rfloor}(1-\tau^2+\tau \lambda_i)^{\lfloor\frac{m-1}{2}\rfloor}.
\end{eqnarray}
Note that  $\{\lambda_i\colon\, 1\le i\le n\}$ and $\{1-\tau^2+\tau\lambda_i\colon\,1\le i\le n\}$ are the spectra of $A_\tau(G)$ and $(1-\tau^2)I_n +\tau A_\tau(G)$, respectively. We  have $\prod_{i=1}^{n}\lambda_i=\det A_\tau(G)$ and  $$\prod_{i=1}^{n}((1-\tau^2) +\tau \lambda_i)=\det ((1-\tau^2)I_n +\tau A_\tau(G)).$$
It follows from Eq. \eqref{pzt} that
\begin{eqnarray}
\prod_{i=1}^n\prod_{j=1}^m Z^{(\tau)}(\mu_i^{(j)})&=&\pm \left(\prod_{i=1}^n \lambda_i\right)^{\lfloor\frac{m}{2}\rfloor}\left(\prod_{i=1}^{n}(1-\tau^2+\tau \lambda_i)\right)^{\lfloor\frac{m-1}{2}\rfloor}\nonumber\\
&=&\pm (\det A_\tau(G))^{\lfloor\frac{m}{2}\rfloor}(\det ((1-\tau^2)I_n+\tau A_\tau(G)))^{\lfloor\frac{m-1}{2}\rfloor}\nonumber.
\end{eqnarray}
This completes the proof.
\end{proof}

\section*{Declaration of competing interest}
There is no conflict of interest.
\section*{Acknowledgments}
This work is supported by the National Natural Science Foundation of China (Grant Nos. 12001006 and 12101379).

\end{document}